\documentclass[12pt,leqno]{amsart}
\usepackage{amsmath,amssymb,amsfonts}
\usepackage{eucal,enumerate}
\usepackage{xypic,graphicx,psfrag}
\usepackage{mathrsfs}
\usepackage{hyperref}

\usepackage{color} %SDC: so xfig's pstex_t file works.

\usepackage{tikz}
\usetikzlibrary{decorations.pathreplacing}
\usepackage{scalefnt}
\newcommand\sectionpage\newpage

\newtheorem{lem}{Lemma}[section]

\newtheorem{prop}[lem]{Proposition}
\newtheorem{thm}[lem]{Theorem}

\theoremstyle{definition}
\newtheorem{conj}[lem]{Conjecture}

\newtheorem{quest}[lem]{Question}

\numberwithin{equation}{section}
\numberwithin{table}{section}
\numberwithin{figure}{section}
\allowdisplaybreaks

\newcommand\vstrut[1]{\rule{0ex}{#1}}

\renewcommand\mod{\, \operatorname{mod}\, }
\renewcommand{\phi}{\varphi} 
\renewcommand{\epsilon}{\varepsilon}

\newcommand\eset{\varnothing}

\newcommand\inv{^{-1}}
\newcommand\transpose{^{\text{T}}}
\newcommand\codim{\operatorname{codim}}

\newcommand\cA{\mathscr{A}}		% script for arrangements, et al.
\newcommand\cB{\mathcal{B}}
\newcommand\cBo{{\cB^\circ}}
\renewcommand\cH{\mathcal{H}}	% \cH seems to be defined by xypic.	\mathcal for subspaces

\renewcommand\cL{\mathscr{L}}	% \cL seems to be defined by xypic.
\newcommand\cP{\mathcal{P}}
\renewcommand\cR{\mathscr{R}}		% \cR seems to be already defined.
\newcommand\cS{\mathscr{S}}
\newcommand\cU{\mathcal{U}}	% \cU for bishops subspace	\mathcal for subspaces
\newcommand\tcU{\widetilde\cU}
\newcommand\cX{\mathcal{X}}
\newcommand\cW{\mathcal{W}}
\newcommand\cY{\mathcal{Y}}
\newcommand\cZ{\mathcal{Z}}
\newcommand\bbQ{\mathbb{Q}}
\newcommand\bbR{\mathbb{R}}
\newcommand\bbZ{\mathbb{Z}}

\newcommand\pB{\mathbb B}	% Bishop
\newcommand\pN{\mathbb N}	% Nightrider
\newcommand\pP{\mathbb P}	% Piece
\newcommand\pQ{\mathbb Q}	% Queen
\newcommand\pR{\mathbb R}	% Rook

\newcommand\bbeta{\boldsymbol\beta}
\newcommand\bz{\mathbf z}
\newcommand\bI{\mathbf I}

\newcommand\lcmd{\operatorname{lcmd}}
\newcommand\lcm{\operatorname{lcm}}

\newcommand\Eta{\mathsf{H}}
\renewcommand\S{\mathsf{\Sigma}}

\newcommand\vol{\operatorname{vol}}
\newcommand\LCM{\mathop{\operatorname{LCM}}}
\newcommand\SM{\operatorname{\mathsf{SM}}}

\newcommand\M{\mathbf{M}}
\newcommand\Kot{Kot\v{e}\v{s}ovec}

\newcommand\Aut{\operatorname{Aut}}

\newcommand\Coeff{\mathcal C}

\newcommand\recurrence{II.7.2}	%{recurrence}
\newcommand\Ctypestwothree{III.4.3}	%{C:types23}

\allowdisplaybreaks

\begin{document}

\pagestyle{myheadings}
\markleft{Chaiken, Hanusa, and Zaslavsky \qquad\qquad \today}
\markright{A $q$-Queens Problem. I.  General Theory \quad \today}

\title{A $q$-Queens Problem  \\
I.  General Theory \\[10pt]  
\today}

\author{Seth Chaiken}
\address{Computer Science Department\\ The University at Albany (SUNY)\\ Albany, NY 12222, U.S.A.}
\email{\tt sdc@cs.albany.edu}

\author{Christopher R.\ H.\ Hanusa}
\address{Department of Mathematics \\ Queens College (CUNY) \\ 65-30 Kissena Blvd. \\ Queens, NY 11367-1597, U.S.A.}
\email{\tt chanusa@qc.cuny.edu}

\author{Thomas Zaslavsky}
\address{Department of Mathematical Sciences\\ Binghamton University (SUNY)\\ Binghamton, NY 13902-6000, U.S.A.}
\email{\tt zaslav@math.binghamton.edu}

\begin{abstract}
By means of the Ehrhart theory of inside-out polytopes we establish a general counting theory for nonattacking placements of chess pieces with unbounded straight-line moves, such as the queen, on a polygonal convex board.  
The number of ways to place $q$ identical nonattacking pieces on a board of variable size $n$ but fixed shape is given by a quasipolynomial function of $n$, of degree $2q$, whose coefficients are polynomials in $q$.  
The number of combinatorially distinct types of nonattacking configuration is the evaluation of our quasipolynomial at $n=-1$.  
The quasipolynomial has an exact formula that depends on a matroid of weighted graphs, which is in turn determined by incidence properties of lines in the real affine plane.  
We study the highest-degree coefficients and also the period of the quasipolynomial, which is needed if the quasipolynomial is to be interpolated from data, and which is bounded by some function, not well understood, of the board and the piece's move directions.

In subsequent parts we specialize to the square board and then to subsets of the queen's moves, and we prove exact formulas (most but not all already known empirically) for small numbers of queens, bishops, and nightriders.  

Each part concludes with open questions, both specialized and broad.
\end{abstract}



%\thanks{Version of \today.}
\thanks{The outer authors thank the very hospitable Isaac Newton Institute for facilitating their work on this project.  The inner author gratefully acknowledges support from PSC-CUNY Research Awards PSCOOC-40-124, PSCREG-41-303, TRADA-42-115, TRADA-43-127, and TRADA-44-168.}

\maketitle

\newpage
\setcounter{tocdepth}{4}
\tableofcontents

%%%%%%%%%%%%%%%%%%%%%%%%%%%%%%%%%%%%%%%%%%%%%%%%%%
\sectionpage\section{Introduction and Preview}\label{intro}

The famous $n$-Queens Problem is to place $n$ nonattacking queens---the largest conceivable number---on an $n\times n$ chessboard, or more broadly, to count the number of placements.  (See, for instance, \cite{Bell} on the former and \cite{8Q} on the latter.)  
The counting problem has no known solution except by individual computation for relatively small values of $n$.  

This article is Part~I of a series \cite{QQ} that presents a natural generalization we call the $q$-Queens Problem, wherein we arbitrarily fix the number of queens, $q$, and vary $n$, the size of the board; also, the ``queen'' may be any of a large class of traditional and fairy chess pieces called ``riders''.  We show (Theorem~\ref{T:formula}) that for each separate problem the number of solutions is, aside from a denominator of $q!$, a quasipolynomial function of $n$, which means it is given by a cyclically repeating sequence of polynomials.  This form of solution could be inferred from empirical formulas for small $q$ found over the decades (see \cite{ChMath}) though it was never proved; our approach makes it obvious.  
Remarkably, the coefficients of this quasipolynomial in $n$ are themselves (up to a normalization) polynomial functions of $q$ (Theorem~\ref{T:gammapoly}); this, too, can be inferred from empirical results, though it seems never to have been considered that making $q$ a variable might give a single comprehensive expression.

Our results apply to any pieces with unbounded straight-line moves, such as the queen, rook, bishop, and the nightrider of fairy chess, which moves arbitrary distances in the directions of a knight's move---in fact, the requirements of our method are the definition of a fairy-chess rider; thus, our proof of quasipolynomiality and coefficient polynomiality applies to all riders, and only to riders.  Our work generalizes in other ways too, as both properties extend to boards of arbitrary rational convex polygonal shape and quasipolynomiality in$n$ extends to mixtures of pieces with different moves.  Here, in Part~I, we develop the theory under the generality of arbitrary rational convex polygonal shapes; however, we restrain the complexity and strengthen the results by assuming all pieces have the same moves.  In Parts~II--V we further narrow the focus successively to square boards (Part~II), to partial queens, whose moves are subsets of the queen's moves, on square boards (Part~III), and then to three special pieces: the bishop, the queen, and the nightrider (Parts~IV and V).  Part~III may be considered the capstone of our series; it is where the theory of the prior parts is applied, many times, to obtain exact results for a narrow but important family of pieces, results which in Part~IV are applied to obtain detailed results about three real (or as one might say of the nightrider, surreal) chess pieces.  Part~V is devoted to one theorem: the exact period for any number of bishops---the only nontrivial period known for all numbers of a single piece.

Setting $q=n$ for queens on the square board gives the first known formula for the $n$-Queens Problem (in Part~II).  It is complex and hard to evaluate except when $q$ is very small, but it is precise and complete.

Our work has two main elements: a method of computation, and a common structural framework for all counting functions for riders.  
The method is that of inside-out polytopes \cite{IOP}, which is an extension of Ehrhart's theory of counting lattice points in convex polytopes (cf.\ \cite[Chapter 4]{EC1}).  The extension adds to a convex polytope an arrangement (a finite set) of forbidden hyperplanes.  The polytope is derived from the board and the hyperplane arrangement expresses the moves.  The lattice of intersection subspaces of the arrangement plays the crucial role in the construction of the counting function.  The proof of quasipolynomiality in $n$ is a simple application of inside-out polytopes.  The proof of bivariate quasipolynomiality is by a more subtle analysis.  
The structural framework, besides proving quasipolynomiality, includes explicit formulas in terms of $q$ for the coefficients of the highest-order powers of $n$ on any board, with stronger results for the square board in Part~II and even stronger ones for partial queens in Part~III, all obtained by careful study of subspaces of low codimension in the hyperplane intersection lattice.  Then Part~IV proves explicit formulas, some new, some known but never rigorously established, for small numbers of the three special pieces mentioned above.  These proofs either apply the general results of Parts~II and III or directly employ inside-out polytope geometry.  Part~V establishes the period for bishops; its technique is exceptional as it employs signed graph theory to assist the inside-out geometry.

\medskip
Now we state the problem more precisely.  It has three ingredients: a piece, a board, and a number.  The \emph{piece} $\pP$\label{d:P} has moves that are all integral multiples of vectors in a nonempty set $\M$\label{d:moveset} of non-zero, non-parallel integral vectors $m_r \in \bbR^2$.\label{d:mr}\label{d:r} 
A \emph{move} is the difference between a new position and the original position; that is, if a piece is in position $z \in \bbZ^2$, it may move to any location $z + \kappa m_r$ for $\kappa \in \bbZ$ and $m_r \in \M$.  
We call the $m_r$'s the \emph{basic moves}.  Each one must be in lowest terms; that is, its two coordinates need to be relatively prime; and no basic move may be a scalar multiple of any other.  (Indeed, the slope of $m_r$ contains all  necessary information and can be specified instead of $m_r$ itself.)  
The \emph{board} consists of the integral points in the interior $(n+1)\cBo$\label{d:n} of a positive integral multiple $t\cB$ of a rational convex polygon $\cB \subset \bbR^2$\label{d:B} (that is, the vertices of $\cB$ have rational coordinates).  The number is $q$\label{d:q}, the number of pieces that are to occupy places on the board.  The rule that no two pieces may attack each other, said mathematically, is that if there are pieces at positions $z_i$ and $z_j$,\label{d:zi} then $z_j-z_i$ is not a multiple of any $m_r$.  

For instance, the polygon may be the unit square $[0,1]^2$.  
The multiple $(n+1) [0,1]^2$ has interior points $(x,y)$ for integers $x,y = 1,2,\ldots,n$.  
A rectangle $[0,a]\times[0,b]$ with positive, rational $a$ and $b$,\label{d:a} whose board is the point set $\big[(0,(n+1)a)\times(0,(n+1)b)\big] \cap \bbZ^2$, is also covered by our work.  
The set $\M$ is $\{(1,1),(1,-1)\}$ for a bishop, $\{(1,0),(1,1),(0,1),(1,-1)\}$ for a queen, and $\{(2,1),(1,2),(2,-1),(1,-2)\}$ for a nightrider.  

\medskip
This is the place to mention the extensive work of Vaclav \Kot, who collected previous results and produced many new formulas to count non-attacking configurations of chess pieces.  His results are reported for instance in his recent book \cite{ChMath} and the related Web site \cite{ChMathWeb}.  \Kot's formulas and numbers were obtained without our theory so our work, to the extent it duplicates his, is an independent confirmation of his results.  More fundamentally, \Kot's method of work usually does not rigorously prove the validity of the formulas; our theoretical work therefore complements his calculations by providing and showing how to provide proofs. 

We took advantage of \Kot's formulas for bishops and queens to guide some of our investigations.  For instance, he found that the quasipolynomial counting formulas tend to have high-degree coefficients that do not vary periodically.  His formulas also suggested that $q$ appears polynomially in each coefficient, aside from a denominator of $q!$.  Those observations led us to more closely examine the polytopal geometry, leading us to a proof of polynomiality with respect to $q$ as well as other results.  

Our results provide a basis for understanding the periodicity properties of the coefficients in \Kot's formulas.  The general Ehrhart theory of inside-out polytopes implies a period that divides the least common multiple of the denominators of the coordinates of certain points.  This least common multiple is called the denominator of the inside-out polytope (see Section~\ref{background}).  In Section~\ref{period} we take an approach involving subdeterminants of matrices to understanding this denominator, but it appears to provide an inefficient bound on the period.  

\medskip
We finally summarize Part I.  It begins in Section~\ref{background} by reviewing inside-out Ehrhart theory.  
Section~\ref{config} defines the hyperplane arrangement that corresponds to a piece and describes simple aspects of the structure of its intersection subspaces---which are the essential ingredient in our approach.  
In Section~\ref{configcount} we prove our main result, a universal structural formula for the counting quasipolynomial, and initiate the theory of its coefficients and their individual periods.  

Then we change focus from the number of nonattacking configurations to their combinatorial structure.  The combinatorial types appear in the geometry of the hyperplane arrangement.  Their number is consequently the evaluation of the main counting function at a board of size $n=-1$ (!).  

Section~\ref{period} tackles the fundamental problem of bounding the quasipolynomial period, or rather its natural upper bound the denominator.  We explain a fairly simple approach from \cite{DKP} involving subdeterminants of matrices.  Examples suggest it provides an inefficient bound; still, some bound is better than none at all.  

In the final section we propose new research directions, among which are pieces of different kinds on the same board, pieces on higher-dimensional boards, and even a wild generalization where the attacking moves depend on which piece is attacked as well as which does the attacking.  
Indeed, throughout the series we list open directions and conjectures.  Two problems seem to be of highest importance.  One that is fundamental to our approach but very complex is that of determining all the subspaces needed to apply our general formulas in examples; this question (see Section~\ref{slope}) lies in the overlap of matroid theory and real incidence geometry.    Then in Section~\recurrence\ we discuss the great dissimilarity between the period of the counting function and the length of a recurrence for its values according to the (largely unproved but highly suggestive) work of \Kot.  An understanding of this phenomenon should permit a vast reduction in the computing power required to get provable formulas.

We end each part with a dictionary of notation for the benefit of the authors and readers.

%%%%%%%%%%%%%%%%%%%%%%%%%%%%
\sectionpage\section{Hyperplanes, Subspaces, and Ehrhart Quasipolynomials}\label{background}

The essential tools for our study are hyperplane arrangements and the Ehrhart theory of inside-out polytopes.  

In a vector space $\bbR^d$, an \emph{arrangement of hyperplanes}, $\cA$,\label{d:cA} is a finite set of hyperplanes, i.e., linear subspaces of codimension 1.  A \emph{region}\label{d:reg} of a hyperplane arrangement is a connected component of the complement of the union of all the hyperplanes.  The \emph{intersection lattice} of $\cA$ is the set 
$$
\cL(\cA) := \big\{ \bigcap \cS : \cS \subseteq \cA \big\},\label{d:cL}
$$
partially ordered by reverse inclusion.  Thus, it is a partially ordered set, it has bottom element $\hat0 = \bbR^d$ and top element $\hat1 = \bigcap\cA$; in fact, it is a geometric lattice.

An \emph{inside-out polytope} $(\cP,\cA)$\label{d:iop} (see \cite{IOP}, which is the source of the following exposition) is a convex polytope $\cP \subseteq \bbR^d$,\label{d:cP} which we assume is closed and full-dimensional, together with a hyperplane arrangement $\cA$ in $\bbR^d$.  
A \emph{region} $\cR$\label{d:cR} of $(\cP,\cA)$ is a nonempty set that is the intersection of $\cP^\circ$, the interior of $\cP$, with a region of the arrangement $\cA$.  
A \emph{vertex} of $(\cP,\cA)$ is any point of $\cP$ that is the intersection of hyperplanes in $\cA$ and boundary hyperplanes of $\cP$; each vertex is the intersection of $k$ linearly independent hyperplanes of $\cA$ with a $k$-dimensional face of $\cP$.  (That includes vertices of $\cP$, for which $k=0$, and any points of intersection of forbidden hyperplanes that lie in the interior of $\cP$, for which $k = d$.)  
When $\cA$ is empty we have just a convex polytope; the vertices are just the vertices of $\cP$.  
The \emph{intersection semilattice} of $(\cP,\cA)$ is the set 
$$
\cL(\cP^\circ,\cA) := \{ \cU \in \cL(\cA) : \cU \cap \cP^\circ \neq \eset \},
\label{d:U}
$$
ordered by reverse inclusion.  
The two sets $\cL(\cA)$ and $\cL(\cP^\circ,\cA)$ are equal for the inside-out polytopes we employ, but in general they can differ.

(In this paper, $\cP$ is $\cB^q$, a $2q$-dimensional polytope that contains all configurations of $q$ pieces in the board, and the hyperplane arrangement is $\cA_\pP$, consisting of hyperplanes that contain all the $2q$-dimensional points representing configurations of $q$ chess pieces $\pP$ in which some pieces attack each other; see the complete definition in Section~\ref{arr}.)

A \emph{quasipolynomial} is a function $f(t)$\label{d:f} of positive integers that can be written in the form $e_d(t)t^d + e_{d-1}(t)t^{d-1} + \cdots + e_0(t)$ where each coefficient $e_j(t)$ is a periodic function of $t$.  The least common multiple $p$ of the periods of all the coefficients is the \emph{period} of $f$.  Another way to describe $f$ is as a function that is given by $p$ polynomials, $f_k(t)$ for $k = 0,1,\ldots,p-1$, under the rule $f(t) = f_k(t)$ if $0 < t \equiv k \mod p$.  
We call the individual polynomials $f_k(t)$ the \emph{constituents} of $f$. 
We say $f$ has \emph{degree $d$} if that is the highest degree of a constituent.  (In our quasipolynomials every constituent has the same degree.)

For a positive integer $t$ and a polytope $\cP$, the number of integer points in $t\cP$,\label{d:t} or equivalently the number of $(1/t)$-fractional points in $\cP$, is denoted by $E_{\cP}(t)$.\label{d:Ehr}  The number in $\cP^\circ$ is denoted by $E_{\cP}^\circ(t)=E_{\cP^\circ}(t)$.\label{d:E}  
We assume the vertices of $\cP$ are rational and we define $D(\cP)$\label{d:D}, the \emph{denominator of $\cP$}, to be the least common denominator of all their coordinates. Then $E_{\cP}^\circ$ is a quasipolynomial function of $t$, the \emph{open Ehrhart quasipolynomial} of $\cP$.  Furthermore, the leading term of every constituent polynomial is $\vol(\cP) t^d$, where the coefficient is the volume of $\cP$,  and the period of this quasipolynomial is a divisor of $D(\cP)$; in particular, if $\cP$ has integral vertices, $E_{\cP}^\circ$ is a polynomial.  (These results are due to Ehrhart; see, e.g., \cite{BR}.)

An inside-out polytope $(\cP,\cA)$ that has rational vertices has similar properties.  
Its \emph{open Ehrhart quasipolynomial} is the function $E_{\cP,\cA}^\circ(t)$\label{d:Eiop} of positive integers $t$ whose value is the number of integer points in the $t$-fold dilate $t\cP^\circ$, or equivalently the number of $(1/t)$-fractional points in $\cP$, that do not lie in any of the hyperplanes of $\cA$.  (The equivalence of the two definitions is due to the fact that homogeneous hyperplanes are invariant under dilation.)  
The \emph{denominator} $D(\cP,\cA)$ is the least common denominator of the coordinates of all vertices.  Given $\cU\in \cL(\cA)$, the \emph{volume} $\vol(\cU\cap\cP)$\label{d:vol} when $\dim \cU < d$ is a relative volume defined in terms of the integral lattice $\cU\cap\bbZ^d$; it is the proportion that the measure of $\cU\cap\cP$ bears to that of a fundamental domain of $\cU\cap\bbZ^d$.  In the case of $\cP$ itself, it is the usual volume, since $\cU = \bbR^d$.

\begin{lem}[{\cite[Theorem~4.1]{IOP}}]\label{L:iop}
The open Ehrhart quasipolynomial has the form 
$$
E_{\cP,\cA}^\circ(t) = e_{d} t^{d} + e_{d-1}(t) t^{d-1} + \cdots + e_0(t) t^0,
$$
where the coefficient $e_d$ is the volume of $\cP$ (a constant) and the coefficients $e_j(t)$ for $j<d$
are periodic functions of $t$ with period that divides the denominator $D(\cP,\cA)$.  
\end{lem}

The period $p$ of $E_{\cP,\cA}^\circ$ equals the least common multiple of the periods of the coefficients.  Thus, $p | D(\cP,\cA)$.  
We write the constituents as $E_{\cP,\cA,i}^\circ$ for $i=0,1,\ldots,p-1$.

A fundamental formula in the Ehrhart theory of inside-out polytopes \cite[Equation~(4.4)]{IOP} is
\begin{equation}\label{E:iopmu}
E_{\cP,\cA}^\circ(t) = \sum_{\cU \in \cL(\cP^\circ,\cA)}  \mu(\hat0,\cU) E_{\cU \cap \cP^\circ}(t),
\end{equation}
where $\mu$ denotes the M\"obius function of $\cL(\cP^\circ,\cA)$.\label{d:mu}  This has the following important consequence.

\begin{lem}\label{L:constantcoefficients}
Suppose $\cU \cap \cP$ has integral vertices for every $\cU \in \cL(\cP^\circ,\cA)$ whose codimension is $< k$.   Then the coefficients $e_{d-i}(t)$ are constant for all $i \leq k$.
\end{lem}

\begin{proof}
Each Ehrhart quasipolynomial on the right-hand side of Equation~\eqref{E:iopmu} has the form 
$$
E_{\cU \cap \cP^\circ}(t) = e_{\dim \cU}(\cU;t) t^{\dim \cU} + e_{\dim \cU-1}(\cU;t) t^{\dim \cU-1} + \cdots + e_0(\cU;t) t^0,
$$
where each $e_j(\cU;t)$ is a periodic function of $t$ and $e_{\dim \cU}(\cU;t)$ is the ($\dim \cU$)-dimensional volume of $\cU \cap \cP$.  If $\cU$ has integral vertices the denominator of $\cU$ is $1$ so each $e_j(\cU;t)$ is $e_j(\cU)$, a constant independent of $t$.  Now, 
\begin{align*}%\label{E:}
E_{\cP,\cA}^\circ(t) &= \sum_{\cU \in \cL(\cP^\circ,\cA)}  \mu(\hat0,\cU) \sum_{j=0}^{\dim \cU} e_j(\cU;t) t^j \\
&= \sum_{j=0}^{d} t^j  \sum_{\substack{\cU \in \cL(\cP^\circ,\cA):\\ \codim \cU \leq d-j}}  \mu(\hat0,\cU) e_j(\cU;t) .
\end{align*}
Thus, 
\begin{equation}\label{E:eCoeff}
e_{d-i}(t) = \sum_{\substack{\cU \in \cL(\cP^\circ,\cA):\\ \codim \cU \leq i}}  \mu(\hat0,\cU) e_{d-i}(\cU;t).
\end{equation}
If $\codim \cU = i$, then $e_{d-i}(\cU;t) = \vol(\cP^\circ \cap \cU)$, a constant independent of $t$.  If also $\cU \cap \cP^\circ$ has integral vertices for all $\cU$ with $\codim \cU < i$, then all coefficients $e_{d-i}(\cU;t) = e_{d-i}(\cU)$, independent of $t$, so $e_{d-i}$ is a constant.  This is true for all $i \leq k$; thus, all terms $t^j$ with $j \geq d - k$ have constant coefficients $e_j(\cU)$.  
\end{proof}

Taking $k=1$ gives a special case of most importance for chess placements. 

\begin{lem}\label{L:2ndcoeff}
If $\cP$ has integral vertices, then $e_{d-1}(t)$ is constant.
\hfill\qed
\end{lem}

%%%%%%%%%%%%%%%%%%%%%%%%%%%%
\sectionpage\section{Configurations}\label{config}

From now on, the polytope is $\cP = \cB^q$, the open polytope is $\cP^\circ=\cBo^q$, and the (open) inside-out polytope is $(\cP^\circ,\cA_\pP)$, where $\cA_\pP$ is the ``move arrangement'' to be defined shortly.  (We assume that $q > 0$.)

%=====
\subsection{Nonattacking configurations and the move arrangement}\label{arr}\

The secret of the solution is to restate each rule of attack as an equation of a forbidden hyperplane in $\bbR^{2q}$.  A \emph{labelled configuration}\label{d:config} describes the locations of $q$ labelled pieces; it is a point $\bz = (z_1,\ldots,z_q) \in \bbR^{2q}$\label{d:bfz} with each $z_i = (x_i,y_i) \in \bbZ^2$.  The labelled configuration is \emph{nonattacking} or \emph{attacking} depending on whether or not it violates every attack equation.  An \emph{attack equation} is a linear constraint on $\bz$ expressing the fact that labelled pieces $\pP_i$ and $\pP_j$ attack each other; in mathematical terms, that $z_j-z_i$ is a multiple of a move $m_r$.  

To express an attack in the configuration space $\bbR^{2q}$, observe that $z_j - z_i \in \langle m_r \rangle$ can be rewritten as $(z_j - z_i) \perp m_r^\perp$, or, $(z_j - z_i) \cdot m_r^\perp = 0,$ where $m_r^\perp$ denotes any nonzero vector orthogonal to $m_r$.  The equation $(z_j - z_i) \cdot m_r^\perp = 0$ is the equation of a hyperplane in the configuration space (the \emph{move hyperplane} $\cH_{ij}^{m_r}$\label{slope-hyp} associated to the move $m_r$) whose points are attacking labelled configurations.  (We also use slope notation: $\cH_{ij}^{d/c}$ when $m_r=(c,d)$, with slope $d/c$\label{d:cd2}.)  
These move hyperplanes in the configuration space form an arrangement of hyperplanes, $\cA_\pP$,\label{d:AP} which we call the \emph{move arrangement} of $\pP$.  There are $\binom{q}{2}|\M|$ of these hyperplanes.  

For specificity, for each basic move vector $m_r = (c_r,d_r)$,\label{d:cd1} we define $m_r^\perp := (d_r,-c_r)$,\label{d:mrperp} which is $m_r$ rotated $90^\circ$ counterclockwise; thus, $m_r^\perp$ points to the left side of the move line.

The fact that the inflated polytope $t \cBo^q$ can engulf arbitrarily large integral points makes our polytopal approach awkward.  
Therefore, we often reduce the integral configuration $\bz  \in t\cB^q$ to a fractional configuration $\bz' = t\inv\bz \in \cB^q$.  The denominators of the components of $\bz'$ tell us which dilates $t\cB^q$ contain a corresponding integral point $t\bz'$, since $t\bz'$ is integral precisely when $t$ is a multiple of the least common denominator of the components of $\bz'$.  We refer to either $\bz \in t\cB^q \cap \bbZ^{2q}$ or $\bz \in \cB^q \cap \bbQ^{2q}$ as a \emph{configuration}, assuming that the context will make clear whether we mean an integral or fractional configuration.

Every move hyperplane contains the diagonal $\{(z,z,\ldots,z) \in \bbR^{2q} : z \in \bbR^2\}$; hence each move hyperplane intersects the interior $\cP^\circ=(\cB^q)^\circ=\cBo^q$ so it is definitely a member of the intersection semilattice.  That is, 
\begin{equation}
\cL(\cBo^q,\cA_\pP)=\cL(\cB^q,\cA_\pP)=\cL(\cA_\pP) \text{ for every board and piece.}
\label{E:arr}
\end{equation}
%

%=====
\subsection{Subspaces in the move arrangement}\label{subspaces}\

In a configuration a piece $\pP_i$ has coordinates $z_i=(x_i,y_i)$.  Each subspace $\cU$ is specified by equations that involve certain of the $q$ labelled pieces, for instance $\pP_1,\ldots,\pP_\kappa$, and no others; then $\cU$ has the form $\widetilde\cU \times \bbR^{2(q-\kappa)}$ where $\tcU$ is a subspace of $\bbR^{2\kappa}$ whose equations in $\bbR^{2\kappa}$ use at least one coordinate corresponding to each of $\pP_1,\ldots,\pP_\kappa$.  We say that $\cU$ has equations that \emph{involve} the pieces $\pP_1,\ldots,\pP_\kappa$ (for short, $\cU$ involves those pieces); and we call $\tcU$\label{d:tcU} the \emph{essential part} of $\cU$.  Let $\cU_\kappa^\nu$ denote any subspace with codimension $\nu$ involving precisely $\kappa$ pieces.\label{d:codim}\label{d:kappa}  Similarly, $\cU_\kappa$ denotes a subspace of any codimension that involves $\kappa$ pieces.

For counting points in $\cP\cap\cU$ we want special notation.  
We define $\alpha(\cU_\kappa;n)$ to be the number of configurations of $\kappa$ pieces in a board of scale factor $n+1$ that satisfy all the attack equations that define $\cU_\kappa$; that is, 
$$
\alpha(\cU;n) := E_{\cBo^\kappa\cap\tcU}(n+1),
\label{d:alphaU}
$$
the number of integral points in $\cB^\kappa \cap \tcU$.  
(We prefer $n=t-1$ as parameter because it is natural for the square board and because we are interested in the coefficients of powers of $n$.)

A subspace $\cU \in \cL(\cA_\pP)$ \emph{decomposes} into subspaces $\cU_1, \cU_2 \in \cL(\cA_\pP)$ if $\cU = \cU_1 \cap \cU_2$ and each piece involved in $\cU$ is involved in just one of $\cU_1$ and $\cU_2$.  
When $\cU$ decomposes into $\cU_1$ and $\cU_2$, then $\dim\cU=\dim\cU_1+\dim\cU_2$ and the interval $[\hat0,\cU]$ has the structure of the product $[\hat0,\cU_1] \times [\hat0,\cU_2]$.  It follows that $\mu(\hat0,\cU) = \mu(\hat0,\cU_1)\mu(\hat0,\cU_2)$.  
Furthermore, 
$$\alpha(\cU;n)=\alpha(\cU_1;n)\alpha(\cU_2;n)$$ 
because $\tcU_1$ and $\tcU_2$ involve different coordinates.  
These reduction formulas are very useful, especially in the treatment of partial queens in Part~III.

Certain subspaces in $\cL(\cA_\pP)$ merit closer examination.  Define\label{d:Wdc}\label{d:W=}
\begin{align*}
\cW_{ij\dots}^{\,d/c} &:= \{ \bz \in \bbR^{2q} : \pP_i, \pP_j, \ldots \text{ all lie in a line with slope } d/c \} 
	\,= \bigcap_{r,s\in\{i,j,\ldots\}}  \cH^{d/c}_{rs},  \\
\intertext{and if $|\M|\geq2$,} 
\cW_{ij\dots}^{\,=} &:= \{ \bz \in \bbR^{2q} :  z_i = z_j = \cdots \}. 
\label{d:W=}
\end{align*}
Then 
\begin{align*}
\codim\cW_{ij\ldots}^{\,d/c} &= \text{number of subscripts}-1, \\
\quad \codim\cW_{ij\dots}^{\,=} &= 2(\text{number of subscripts}-1),
\end{align*}
and, for instance, 
$$
\cW_{ij}^{\,=} = \bigcap_{(c,d)\in\M} \cH^{d/c}_{ij} = \cH^{d/c}_{ij} \cap \cH^{d'/c'}_{ij}
$$
for any two distinct slopes $d/c$ and $d'/c'$.

\begin{lem}\label{L:Wmu}
For any board and any move set $\M$ with $|\M|\geq2$, we have M\"obius functions 
\begin{align*}
\mu(\hat0,\cW_{i_1\ldots i_l}^{\,d/c}) &= (-1)^{l-1}(l-1)! , \\
\mu(\hat0,\cW_{ij}^{\,=}) &= |\M|-1, \\
\mu(\hat0,\cW_{ij}^{\,=} \cap \cW_{ijk}^{d/c}) &= -2(|\M|-1), \\
\mu(\hat0,\cW_{ijk}^{\,=}) &= (|\M|-1)^2(|\M|-3).
%\label{E:W mu}
\end{align*}
\end{lem}

\begin{proof}
The interval $[\hat0,\cW_{i_1\ldots i_l}^{\,d/c}]$ is isomorphic to the partition lattice of $[l]$, or the lattice of flats of the complete graph $K_l$, because $\cH_{ij}^{d/c} \cap \cH_{jk}^{d/c} \subset \cH_{ik}^{d/c}$.  This M\"obius function is well known (see, e.g., \cite{EC1}).

The computation of $\mu(\hat0,\cW_{ij}^{\,=})$ is routine.  That of $\mu(\hat0,\cW_{ij}^{\,=} \cap \cW_{ijk}^{\,d/c})$ is demonstrated in Figure~\ref{F:W12=3 mu}.   
\begin{figure}[htb]
\begin{center}
\quad
\xymatrix{
&&&\underset{-2(|\M|-1)}{\cW_{12}^{\,=} \cap \cW_{123}^{\,d/c}}\ar@{-}[dlll]\ar@{-}[dll]\ar@{-}[dl]\ar@{-}[d]\ar@{-}[dr]\ar@{-}[drr]	&&\\
\underset{(2)}{\cW_{123}^{\,m_1}}\ar@{-}[d]\ar@{-}[dr]\ar@{-}[drr]
	&\underset{(1)}{\cH_{13}^{m_1}\cap\cH_{12}^{m_2}}\ar@{-}[dl]\ar@{-}[drr]	&\underset{(1)}{\cH_{23}^{m_1}\cap\cH_{12}^{m_2}}\ar@{-}[dl]\ar@{-}[dr]
	&\cdots	&\underset{(1)}{\cH_{23}^{m_1}\cap\cH_{12}^{m_{|\M|}}}\ar@{-}[dlll]\ar@{-}[dr]
	&\underset{(|\M|-1)}{\cW_{12}^{\,=}}\ar@{-}[dlll]\ar@{-}[dll]\ar@{-}[dl]\ar@{-}[d]	\\
\underset{(-1)}{\cH_{13}^{m_1}}	&\underset{(-1)}{\cH_{23}^{m_1}}
	&\underset{(-1)}{\cH_{12}^{m_1}}	&\underset{(-1)}{\cH_{12}^{m_2}}
	&\cdots	&\underset{(-1)}{\cH_{12}^{m_{|\M|}}}	\\
&&&\underset{(1)}{\bbR^{2q-6}}\ar@{-}[ulll]\ar@{-}[ull]\ar@{-}[ul]\ar@{-}[u]\ar@{-}[ur]\ar@{-}[urr]	&&
}
\caption{The computation of $\mu(\hat0,\cW_{12}^{\,=} \cap \cW_{123}^{\,d/c})$, exhibited in the Hasse diagram of the interval $[\hat0,\cW_{12}^{\,=} \cap \cW_{123}^{\,d/c}]$ in $\cL(\cA_\pP)$, with the value of $\mu(\hat0,\cU)$ below the subspace $\cU$.  We let $m_1=d/c$.}
\label{F:W12=3 mu}
\end{center}
\end{figure}
That of $\mu(\hat0,\cW_{ijk}^{\,=})$ is shown in Figure~\ref{F:W123= mu}.
\end{proof}

\begin{figure}[htb]
\begin{center}
\quad
\xymatrix{
	&&\underset{(|\M|-1)^2(|\M|-3)}{\cW_{123}^{\,=}}\ar@{-}[dl]\ar@{-}[dr]
	&&
\\
	&\underset{(1)}{\cH_{ij}^{d/c}\cap\cH_{jk}^{d'/c'}\cap\cH_{ik}^{d''/c''}}\ar@{-}[d]
	&&\underset{-2(|\M|-1)}{\cW_{ij}^{\,=}\cap\cW_{123}^{\,d/c}}\ar@{-}[dl]\ar@{-}[dr]
	&
\\
	&\underset{(1)}{\cH_{ij}^{d/c}\cap\cH_{jk}^{d'/c'}}\ar@{-}[dl]\ar@{-}[drr]
	&\underset{(2)}{\cW_{123}^{\,d/c}}\ar@{-}[dl]\ar@{-}[d]\ar@{-}[dr]
	&&\underset{(|\M|-1)}{\cW_{ij}^{\,=}}\ar@{-}[dl]\ar@{-}[d]
\\
	\underset{(-1)}{\cH_{jk}^{d'/c'}}
	&\underset{(-1)}{\cH_{ik}^{d/c}}
	&\underset{(-1)}{\cH_{jk}^{d/c}}
	&\underset{(-1)}{\cH_{ij}^{d/c}}
	&\underset{(-1)}{\cH_{ij}^{d^*/c^*}}	
\\
	&&\underset{(1)}{\bbR^{2q-6}}\ar@{-}[ull]\ar@{-}[ul]\ar@{-}[u]\ar@{-}[ur]\ar@{-}[urr]
	&&
}
\caption{The computation of $\mu(\hat0,\cW_{123}^{\,=})$, exhibited in a partial Hasse diagram of the interval $[\hat0,\cW_{123}^{\,=}]$.  The value of $\mu(\hat0,\cU)$ sits under the subspace $\cU$.  Slopes $d/c, d'/c', \ldots$ are assumed distinct.  There are $3|\M|$ hyperplanes, $3$ subspaces $\cW_{ij}^{\,=}$ for $i,j\in[3]$, $|\M|$ of type $\cW_{123}^{\,d/c}$, $3(|\M|)_2$ of form $\cH_{ij}^{d/c}\cap\cH_{jk}^{d'/c'}$, $(|\M|)_3$ of form ${\cH_{ij}^{d/c}\cap\cH_{jk}^{d'/c'}\cap\cH_{ik}^{d''/c''}}$, $3|\M|$ of type ${\cW_{ij}^{\,=}\cap\cW_{123}^{\,d/c}}$.}
\label{F:W123= mu}
\end{center}
\end{figure}
%

%=====
\subsection{The slope graph}\label{SlopeGraph}\

The \emph{slope graph} $\S(\cU)$\label{d:S} is a labelled graph associated with an intersection subspace $\cU$.  The subspace involves certain pieces; the nodes of $\S(\cU)$ correspond to those pieces and the edges correspond to all the hyperplanes that contain $\cU$.  
The edge corresponding to a hyperplane $\cH^{d/c}_{ij}$ is labelled by the slope; thus we call the edge $e^{d/c}_{ij}$.  An isomorphism of slope graphs is defined as a graph isomorphism that preserves slope labels.
It is clear that subspaces are isomorphic if and only if their slope graphs are isomorphic.  (That is why we call subspace isomorphism ``combinatorial''.)  Note that to specify $\cU$ it is not necessary to state all edges; only $\codim\cU$ edges are necessary.

The slope graph of a subspace is a subgraph of the slope graph of the arrangement $\cA_\pP^q$.  (The superscript in $\cA_\pP^q$ is a reminder that we are in $\bbR^{2q}$.)  Let $\S_q(\pP)$ be the graph with nodes corresponding to $q$ pieces labelled $1,2,\ldots,q$ and edges corresponding to the hyperplanes of $\cA_\pP^q$.  Let $\S(\pP)$ be the union of all these; its node set is $\{v_i: i=1,2,3,\ldots\}$.  We use $\S_q(\pB)$, the bishop slope graph, to prove in Part~V that the period of a bishops counting function is at most 2, and we hope it can be used to get general results about periods.  (Because the slope graph $\S(\pP)$ corresponds to hyperplane arrangements there is a related finitary matroid closure on the edge set whose closed sets are the edge sets of subspace graphs $\S(\cU)$, but it is complicated and we do not need it here.  See the problem statement in Section~\ref{slope}.)

%=====
\subsection{Isomorphic subspaces}\label{isom}\

We now analyze the combinatorial structure of the subspaces $\cU \in \cL(\cA_\pP^q)$.  

Given $\cU=\cU_\kappa$ that involves pieces $\pP_1,\ldots,\pP_\kappa$, let $\Aut(\cU)$\label{d:Aut} be the group of permutations of those pieces---that is, permutations of indices of the coordinates $z_i$---that leave $\cU$ invariant.  We call such a permutation an \emph{automorphism} of $\cU$.  Then there are $\kappa!/|\Aut(\cU)|$ subspaces that are \emph{similar} to $\cU$, meaning that they have the same equations except for a permutation of the subscripts $1,\ldots,\kappa$ in the coordinates.  
For instance, let $\cX$ be determined by $x_1=x_2$, $y_1=y_2$, $x_1-x_3=y_1-y_3$, $\cY$ by $x_1=x_3$, $y_1=y_3$, $x_1-x_2=y_1-y_2$, and $\cZ$ by $x_1+y_1=x_2+y_2=x_3+y_3$, $x_1-y_1=x_3-y_3$.  All have the form $\cU_3^3$ but $\cX$ and $\cY$ are similar while $\cZ$ is dissimilar to the others.  

We say two subspaces $\cU, \cU' \in \cL(\cA_\pP)$ are \emph{isomorphic} or \emph{have the same type} if there is a bijection from the pieces involved in $\cU$ to those involved in $\cU'$ which transforms the equations of $\cU$ to those of $\cU'$.  (Defined precisely by way of the slope graph $\S(\cU)$ in Section~\ref{SlopeGraph}.)  We write $[\cU]$ for the isomorphism class, or \emph{type}, of $\cU$.  The subspaces need not live in the same dimension; i.e., if $\cU\subseteq\bbR^{2q}$ and $\cU'\subseteq\bbR^{2q'}$, $q$ and $q'$ need not be equal (unless we are restricting ourselves to $\cL(\cA_\pP^q)$).  Isomorphism is concerned only with the pieces that are involved in the subspaces.  Isomorphic subspaces have the same codimension $\nu$ (because $\codim\tcU$ in $\bbR^{2\kappa}$ equals $\codim\cU$ in $\bbR^{2q}$) though not necessarily the same dimension, and the same M\"obius function $\mu(\hat0,\cU)$, counting function $\alpha(\cU;n)$ (since it is the Ehrhart quasipolynomial of $\tcU$), and automorphism group.
The number of subspaces in $\cL(\cA_\pP^q)$ that are isomorphic to $\cU_\kappa$ is $\binom{q}{\kappa}\cdot\kappa!/|\Aut(\cU_\kappa)|$.  

Isomorphism is a combinatorial relation.  That can be made precise through the slope graph $\S(\cU)$ since it is clear that $\cU$ and $\cU'$ are isomorphic if and only if there is an isomorphism of their slope graphs that preserves edge labels.  It is then obvious that $\cU$ and $\tcU$ are isomorphic, and that $\cU$ and $\cU'$ are isomorphic if and only $\tcU$ and $\tcU'$ are isomorphic.

Taking account of isomorphism and the fact that
$$
E_{\cU_\kappa\cap\cBo^q}(n+1) = \alpha(\cU_\kappa;n) N^{q-\kappa},
$$
Equation~\eqref{E:iopmu} can be rewritten as
\begin{equation}
\begin{aligned}
E_{\cP,\cA}^\circ(n+1) 
&= \sum_{\kappa=0}^q (q)_\kappa  \sum_{[\cU_\kappa]} \mu(\hat0,\cU_\kappa) \frac{1}{|\Aut(\cU_\kappa)|}  \alpha(\cU_\kappa;n) N^{q-\kappa} \\
&= \sum_{\kappa=0}^\infty (q)_\kappa  \sum_{[\cU_\kappa]} \mu(\hat0,\cU_\kappa) \frac{1}{|\Aut(\cU_\kappa)|}  \alpha(\cU_\kappa;n) N^{q-\kappa} ,
\end{aligned}
\label{E:iopmuisom}
\end{equation}
where the inner sum ranges over all subspace types $[\cU_\kappa]$, since the limit $q$ on $\kappa$ is equivalent to saying there is a representative $\cU_\kappa \in \cL(\cA_\pP^q)$.  
Letting $\kappa$ range up to $\infty$ does not change the value of the expression because $(q)_\kappa=0$ when $\kappa>q$.

%=====
\sectionpage\section{Counting Configurations }\label{configcount}\

An \emph{unlabelled configuration} is a multiset of planar points $z_1,\ldots,z_q$; it corresponds to having unlabelled pieces.  Unlabelled configurations are what we really want to count.  
Since a point with any $z_j = z_i$ is attacking, the same number $q!$ of nonattacking unlabelled configurations corresponds to each nonattacking labelled configuration.

%=====
\subsection{The number of configurations }\label{count}\

Let $u_\pP(q;n)$\label{d:nonattacking} be the number of nonattacking unlabelled configurations and $o_\pP(q;n)$ the number of nonattacking labelled configurations with $q$ pieces in the interior of the dilated board $t\cB$, where $t=n+1$ is a positive integer; thus $u_\pP(q;n) = o_\pP(q;n)/q!$.

Let $\vol\cB$ denote the area of $\cB$.  

\begin{thm}\label{T:formula}
For each positive integer $q$, the number $u_\pP(q;n)$ of nonattacking unlabelled configurations of $q$ pieces in $t\cB$ is given by a quasipolynomial function of $t$, of degree $2q$ with leading coefficient $(\vol\cB)^q/q!$.  
\end{thm}

\begin{proof}
We prove the theorem by showing that $o_\pP(q;n)$ is a quasipolynomial with suitable properties.

We already know that $o_\pP(q;n)$ is the number of integral lattice points in the interior of $t\cB^q$ that are not in any of the move hyperplanes.  Since the hyperplanes are homogeneous, the integral lattice points in $t\cB^q$ can be scaled to $t\inv$-fractional points in $\cB^q$.  Technically, inside-out theory applies to the count of these fractional points in $\cB^q$.  
The theory also requires that the move hyperplanes have rational equations, which they do.  According to the theory, $o_\pP(q;n)$ is a quasipolynomial function of $t$ whose degree is $\dim \cB^q$, which is $2q$, and whose leading coefficient is the volume of $\cB^q$, which is $\vol(\cB)^q$.
\end{proof}

We expand the counting function as a quasipolynomial:
\begin{equation}
u_\pP(q;n) = \gamma_0(n) n^{2q} + \gamma_1(n) n^{2q-1} + \gamma_2(n) n^{2q-2} + \cdots + \gamma_{2q}(n) n^0.  
\label{d:gamma}
\end{equation}
Its leading coefficient is a constant, $\gamma_0(n) = \vol(\cB)^q/q!$ (since $u_\pP(q;n)$ is a disguised Ehrhart quasipolynomial) and its period is a divisor of the denominator $D(\cB^q,\cA_\pP)$.

%=====
\subsection{The general form of coefficients}\label{coeffs}\

We demonstrate that the coefficients in the labelled counting quasipolynomial $o_\pP(q;n)=q!u_\pP(q;n)$ are themselves polynomials in $q$.  
The key to the proof is a new variable, 
$$
N := E_\cBo(n+1),\label{d:N}
$$
the number of lattice points in the interior of the dilated board; that is, it is the number of locations a single piece can be placed on the dilated board.  This variable is a quadratic quasipolynomial function of $n$ (and its leading term is $(\vol\cB) n^2$), so in terms of $N$ we can expand each subspace Ehrhart quasipolynomial as
$$
\alpha(\cU_\kappa^\nu;n) = \sum_{j=0}^{\kappa-\lceil\nu/2\rceil} \bar\Gamma_j(\cU_\kappa^\nu) N^{\kappa-\lceil\nu/2\rceil-j},
$$
where $\bar\Gamma_j = n\bar\Gamma_{j1} + \bar\Gamma_{j0}$,\label{d:bG} a linear quasipolynomial in $n$; thus the Ehrhart quasipolynomial of $\cU_\kappa^\nu$ becomes
\begin{equation}
E_{\cBo^q\cap\cU_\kappa^\nu}(n+1) = \alpha(\cU_\kappa^\nu;n) N^{q-\kappa} = \sum_{j=0}^{\kappa-\lceil\nu/2\rceil} \bar\Gamma_j(\cU_\kappa^\nu) N^{q-\lceil\nu/2\rceil-j}.
\label{E:subspaceN}
\end{equation}

Write $\Coeff_k(f(n))$\label{d:cC} for the coefficient of $n^k$ in a quasipolynomial $f(n)$ (this coefficient may vary periodically with $n$).

\begin{thm}[Coefficient Theorem]\label{T:gammapoly}
The coefficient $q!\gamma_i$ of $n^{2q-i}$ in each constituent of $o_\pP(q;n)$ is a polynomial in $q$ of degree $2i$.  A formula for the coefficient is 
\begin{equation}
\begin{aligned}
q!\gamma_i 
&= \Coeff_{2q-i}(N^q) \\
&\quad+ \sum_{\kappa=\max(2,\lceil i/2\rceil)}^{2i}  (q)_\kappa  \sum_{\nu=\lceil\kappa/2\rceil}^{\min(i,2\kappa-2)} 
\sum_{[\cU_\kappa^\nu]}  \mu(\hat0,\cU_\kappa^\nu) \frac{1}{|\Aut(\cU_\kappa^\nu)|} \, \Coeff_{2q-i}\Big( \sum_{k=\lceil\nu/2\rceil}^\kappa  \bar\Gamma_{k-\lceil\nu/2\rceil}(\cU_\kappa^\nu) N^{q-k} \Big).
\label{E:gammasum}
\end{aligned}
\end{equation}
In particular, $\gamma_0 = (\vol\cB)^q/q!$, a constant, and $\gamma_1$ is constant if the vertices of $\cB$ are integral points.
\end{thm}

\begin{proof}
The constancy of $\gamma_1$ when $\cB$ is integral is from Lemma~\ref{L:2ndcoeff}.

We rewrite Equation~\eqref{E:iopmuisom} via Equation~\eqref{E:subspaceN} as
\begin{equation}
o_\pP(q;n) 
= \sum_{\kappa=0}^\infty \sum_{\nu\geq0} \sum_{[\cU_\kappa^\nu]}  (q)_\kappa \mu(\hat0,\cU_\kappa^\nu) \frac{1}{|\Aut(\cU_\kappa^\nu)|} 
\sum_{j=0}^{\kappa-\lceil\nu/2\rceil} \bar\Gamma_j(\cU_\kappa^\nu) N^{q-\lceil\nu/2\rceil-j} ,
\label{E:typesum1}
\end{equation}
where the third summation ranges over all subspace types $[\cU_\kappa^\nu]$ for $\kappa \leq q$.  In this sum only the factors $(q)_\kappa$ and $N^q$ involve $q$.  

Since no subspace can involve only one piece, there are two kinds of subspace:  $\cU_{2q}^0=\bbR^{2q}$, which gives the term $N^q$, and $\cU_\kappa^\nu$ with $\nu>0$ and $\kappa\geq2$.
If we define $\bar\Gamma_j:=0$ when $j<0$ and note that only the subspace $\bbR^{2q}$ contributes an $N^q$ term, then we may write 
\begin{align*}
\begin{aligned}
o_\pP(q;n) 
&= N^q + \sum_{\kappa=2}^\infty \sum_{\nu=0}^\infty \sum_{[\cU_\kappa^\nu]}  (q)_\kappa \mu(\hat0,\cU_\kappa^\nu) \frac{1}{|\Aut(\cU_\kappa^\nu)|} 
\sum_{j=0}^{\kappa-\lceil\nu/2\rceil} \bar\Gamma_j(\cU_\kappa^\nu) N^{q-\lceil\nu/2\rceil-j} 
\\
&= N^q + \sum_{\kappa=2}^\infty  (q)_\kappa  \sum_{\nu=0}^{\infty} \sum_{[\cU_\kappa^\nu]}  \mu(\hat0,\cU_\kappa^\nu) \frac{1}{|\Aut(\cU_\kappa^\nu)|}  \sum_{k=\lceil\nu/2\rceil}^\kappa  \bar\Gamma_{k-\lceil\nu/2\rceil}(\cU_\kappa^\nu) N^{q-k}.
\label{E:typesum3}
\end{aligned}
\end{align*}
Substituting for $N$ and the $\bar\Gamma_j$ in terms of $n$, the last expression becomes a quasipolynomial function of $n$.  
We conclude that $q!\gamma_i$ is the coefficient of $n^{2q-i}$ in that quasipolynomial.  

This coefficient is independent of $q$ except for the factor $(q)_\kappa$.  It is also a finite sum.  
First, $\nu\leq i$ since a subspace of codimension greater than $i$ will not contribute to the coefficient of $n^{2q-i}$.  Also, $\nu \geq \kappa/2$ because every piece must participate in an equation of $\cU_\kappa^\nu$, each of which involves two pieces.  It follows from $\kappa/2\leq\nu\leq i$ that $\kappa\leq2i$.  In addition, $k \leq \kappa \leq 2i$.  
Since $q$ appears only in the factors $(q)_\kappa$ and $\kappa$ is bounded by $2i$, we conclude that $q!\gamma_i$ is a polynomial function of $q$ of degree at most $2i$ if $n$ is held constant.  

The bound $\nu \leq 2\kappa-2$ follows from the fact that the smallest subspace that involves $\kappa$ pieces is $\cW^{\,=}_{[\kappa]}$, which has codimension $2\kappa-2$.

Moreover, $\kappa\geq i/2$.  For $\cU=\cU_\kappa^\nu$, the dimension of $\tcU$ is $2\kappa-\nu$ so the leading term of $\alpha(\cU;n)$ has degree $2\kappa-nu$ and the last term has degree $0$ in terms of $n$.  For $E_{\cBo^q\cap\cU}(n+1) = \alpha(\cU;n)N^{q-\kappa}$ has, in terms of $n$, leading and final degrees $2q-\nu$ and $2q-2\kappa$.  That is, if $i>2\kappa$, $\cU$ cannot contribute to $\gamma_i$; that is why we may assume the restriction $\kappa\geq i/2$.

Now we consider the effect of the variability of $n$.  Let $P$ be the smallest common period (with respect to $n$) of all $\alpha(\cU;n)$ for all types $[\cU]$ that appear in the sum \eqref{E:gammasum}.  Every $\bar\Gamma_{i-\nu}(\cU)$ is periodic with a period that divides $P$, and so is $N$.  Therefore the coefficient of $(q)_\kappa$ in \eqref{E:gammasum} is periodic with a period that divides $P$.  We conclude that $q!\gamma_i$ is a polynomial function of $q$ whose precise polynomial may vary periodically with $n$ and whose period divides $P$. 

The last step is to determine the degree.  We showed that the largest conceivable degree of $q$ is $\kappa=2i$.  A term with $(q)_{2i}$ arises only from a subspace involving $\kappa=2i$ pieces; then $\nu$ must equal $i$.  The degree is indeed $2i$ because a coefficient of $(q)_{2i}$ arises only from a subspace of the form $\cU_{2i}^i$.  
That coefficient is $\bar\Gamma_0(\cU_{2i}^i)$, which is the leading coefficient of $\alpha(\cU_{2i}^i;n)$, hence a volume and not zero.  In sum, the coefficient of $(q)_{2i}$ in $q!\gamma_i$ is nonzero so $q!\gamma_i$ has degree exactly $2i$.
\end{proof}

%=====
%%%%%%%%%%%%%%%%%%%%%%%%%%%%
\sectionpage\section{Configuration types }\label{types}\

For each basic move $m_r \in \M$ from a fixed location, the line $\langle m_r \rangle$ is naturally directed, so it has a left and right side.  Given a nonattacking configuration of $q$ pieces, record, for each piece $\pP_i$ at location $z_i$ and each move line $z_i + \langle m_r \rangle$ through $\pP_i$, oriented in the direction of $m_r$, the indices of the points $\pP_j$ that lie on the left side of the line.  The set of these lists, for every pair, is the \emph{combinatorial type} of the labelled configuration, briefly the \emph{labelled configuration type}.  
(Our configuration types are always nonattacking; we are not interested in attacking types.)  

Another way to describe a labelled configuration type is by building a nonattacking configuration, one piece at a time.  Place the first piece, labelled $\pP_1$, at $z_1 \in t\cBo \cap \bbZ^2$.  This creates $|\M|$ lines through $z_1$ that may not be occupied by any other piece, and $2|\M|$ regions that may be occupied.  The second piece, $\pP_2$ at $z_2$, will be in one of these regions.  
Now we have $|\M|$ forbidden lines through $z_2$; these lines in combination create permitted regions that will be occupied by the other pieces.  Placing the third piece further subdivides these into a larger number of regions, and similarly with each placement up to the last piece.  
The sequence of choices of region is equivalent to the labelled configuration type.  

By forgetting the order of the pieces we have an \emph{unlabelled combinatorial type} of configuration, for short an \emph{unlabelled configuration type}.  
We want to know how many unlabelled configuration types there are.  

\begin{lem}\label{L:typetypes}
There are $q!$ labelled nonattacking configuration types for each unlabelled type.
\end{lem}

\begin{proof}
In the left-to-right direction $-m_r^\perp$ perpendicular to a move line $\langle m_r \rangle$, the $q$ labelled pieces appear in a definite order, $(\pP_1,\pP_2,\ldots,\pP_q)$.  This is indicated by the left-side list of the $i$th piece with respect to $m_r$, which is $\{ 1, \ldots, {i-1} \}$.  Renumbering the pieces changes the order, hence the left-side list of at least one piece, and therefore the labelled configuration type.
\end{proof}

\begin{lem}\label{L:regionstypes}
The labelled nonattacking configuration types are in one-to-one correspondence with the regions of $(\cB^q,\cA_\pP)$.
\end{lem}

\begin{proof}
We need to be more formal about labelled configuration types.  A configuration on the $t$-fold board $t\cB$, in terms of coordinates, is $\bz = (z_1,z_2,\ldots,z_q) \in (t\cBo)^q \cap \bbZ^{2q}$.  We normalize this to a fractional configuration $t\inv \bz \in (\cBo)^q$.  The normalized configurations are just as good as the integral ones as far as describing configuration types, because the inequalities that describe types are homogeneous.

The type of a configuration $\bz$ is a list of lists of lists.  The $i$th point has lists\label{d:Lir} $L_{ir} = \{ j: \pP_j \text{ is on the left side of the $r$th move line through } \pP_i \},$ one for each basic move $m_r$.  That is, 
\[
L_{ir} = \{ j: (z_j - z_i) \cdot m_r^\perp > 0 \}.
\]
This inequality is precisely what defines the positive halfspace of a hyperplane in $\cA_\pP$; the collection of all such inequalities derived from the configuration $\bz$ determines a subset of the interior of $\cB^q$, which is nonvoid because it contains the fractional configuration $t\inv \bz $.  Consequently, a configuration determines a region of $(\cB^q,\cA_\pP)$.

Conversely, any region contains a fractional point $t\inv \bz \in t\inv\bbZ^{2q}$ for a sufficiently large integer $t$.  Therefore, it corresponds to one or more configuration types.  However, it cannot correspond to more than one configuration type, because the inequalities that define the region determine which indices $j$ are in which list $L_{ir}$ for each $i$ and $r$.
\end{proof}

Recall that $u_\pP(q;n)$ and $o_\pP(q;n)$ are quasipolynomials whose constituents are $u_{\pP,0}(q;n), \ldots,$ $u_{\pP,p-1}(q;n)$ and $o_{\pP,0}(q;n), \ldots,o_{\pP,p-1}(q;n)$, respectively.  The last constituents may be also written $u_{\pP,-1}(q;n)$ and $o_{\pP,-1}(q;n)$ due to the periodicity.

\begin{thm}\label{T:typenumber}
The number of unlabelled combinatorial types of nonattacking configuration of $q$ pieces $\pP$ equals $u_{\pP,-1}(q;-1)$.  The number of labelled nonattacking configuration types equals $o_{\pP,-1}(q;-1)$.  
\end{thm}

\begin{proof}
By \cite[Theorem 4.1]{IOP}, the number of regions of $(\cB^q,\cA_\pP)$ is 
$$E_{\cB^q,\cA_\pP}(0) = (-1)^{\dim(\cB^q)} E^\circ_{\cB^q,\cA_\pP}(0) = E^\circ_{\cB^q,\cA_\pP}(0),$$ 
which in terms of $n$ is $o_{\pP,-1}(q;-1)$.
\end{proof}

Still another way to look at the type of a configuration is through isotopy.  Two labelled configurations $\bz$ and $\bz'$ are \emph{isotopic} if one can be deformed into the other by a continuous movement in the configuration space $\bbR^{2q}$\label{d:configsp} without at any time crossing a move hyperplane.  It is clear from the correspondence between lists and regions in the proof of Lemma~\ref{L:regionstypes}, and the convexity of regions, that this is possible if and only if $\bz$ and $\bz'$ have the same left-side lists, and then the isotopism can be performed along a line segment in the interior of $t\cB^q$.

On the other hand, one might ask about discrete isotopy, where we move one piece at a time on the board.  A \emph{discrete isotopism} is a sequence of steps, $\bz = \bz^0 \to \bz^1 \to \cdots \to \bz^k = \bz'$ where $\bz^{j-1}$ and $\bz^j$ differ only by making a legitimate move of a single piece that does not change the combinatorial type of the configuration.  One should allow any amount of inflation, i.e., one should be allowed to multiply all the coordinates by a very large positive integer before performing the isotopism (which one can think of as replacing the $(1/t)$-lattice by a finer $(1/kt)$-lattice).  One would naively expect this to be equivalent to continuous isotopy, and indeed it is, once we overcome two difficutlies.

First, if there is only one basic move $m$, the only configurations that can be reached from a nonattacking configuration $\bz$ by allowed moves are those in the line through $\bz$ in the direction of $m$ (in some $z_i$ plane).  Therefore, there must be more than one move in $\M$.  

Second, even if there are two basic moves, there can be configurations that are unreachable from each other in a board $t\cB$ if $t$ is fixed; for instance, in a configuration of bishops no moves can change the numbers of bishops on squares of each color.  That problem is solvable by inflation.

\begin{thm}\label{T:isotopyequiv}
If $|\M| \geq 2$, isotopy and discrete isotopy produce the same equivalence relation on nonattacking configurations.
\end{thm}

The proof begins with a planar lemma.  Note that any two basic moves are nonparallel.

\begin{lem}\label{L:planarmoves}
Given two basic moves, $m_1$ and $m_2$, there is a sequence of moves that takes a piece from $z=(x,y) \in \bbZ^2$ to $z'=(x',y') \in \bbZ^2$ if and only if $(x'-x,y'-y)$ is divisible by $\det(m_1,m_2)$.
\end{lem} 

\begin{proof}
There exists such a sequence if and only if there is an integral solution to $\kappa m_1 + \lambda m_2 = (x'-x,y'-y)$.  Let $C = \begin{pmatrix} m_1 &m_2 \end{pmatrix}$.\label{d:C}  The equation to be solved is $\begin{pmatrix} \kappa \\ \lambda \end{pmatrix} C = \begin{pmatrix} x'-x \\ y'-y \end{pmatrix}$.  Inverting, 
$$
\begin{pmatrix} \kappa \\ \lambda \end{pmatrix} = (\det C)\inv C^* \begin{pmatrix} x'-x \\ y'-y \end{pmatrix},
$$ 
where $C^*$ is the cofactor matrix, which is an integral matrix.  By the assumptions on moves, the greatest common divisor of the entries in either column of $C$ is 1; thus, an integral solution exists if and only if $x'-x$ and $y'-y$ are multiples of $\det C$.
\end{proof}

\begin{proof}[Proof of Theorem~\ref{T:isotopyequiv}]
Choose two basic moves $m_1, m_2 \in \M$.

Let $\bz=(z_1,\ldots,z_q),\ \bz'=(z'_1,\ldots,z'_q) \in \bbZ^{2q}$ be nonattacking configurations of the same combinatorial type on the dilated board $t\cBo$, where $t>0$.
No restriction hyperplane separates them; that is, they lie in the same open region $\cR$ of $\cA_\pP$.  We want a discrete isotopism from $\bz$ to $\bz'$, that is, is a sequence of moves of individual pieces that gives a sequence of (fractional) configurations lying in $\cR \cap \cBo \cap \tau\inv\bbZ^{2q}$ for some $\tau \in \bbZ_{>0}$.\label{d:tau}  

A sequence of individual moves is expressed (disregarding its order) by solving the Diophantine equation 
$$
\kappa_1 \begin{pmatrix} m_1 \\ 0 \\ \vdots \\ 0 \end{pmatrix} + \cdots + \kappa_q \begin{pmatrix} 0 \\ \vdots \\ 0 \\ m_1 \end{pmatrix} 
+ \lambda_1 \begin{pmatrix} m_2 \\ 0 \\ \vdots \\ 0 \end{pmatrix} + \cdots + \lambda_q \begin{pmatrix} 0 \\ \vdots \\ 0 \\ m_2 \end{pmatrix} 
= \begin{pmatrix} \bz'_1-\bz_1 \\ \bz'_2-\bz_2 \\ \vdots \\ \bz'_q-\bz_q \end{pmatrix} .
$$
This is a set of $q$ independent equations:
$$
\kappa_i m_1 + \lambda_i m_2 = z'_i-z_i  
$$
for $i = 1,\ldots, q$.  We know from Lemma~\ref{L:planarmoves} when they are solvable:  if and only if $\det C$ divides every component of $\bz' - \bz$.  If we multiply the entire board by $k$, replacing $\bz' - \bz$ by $k(\bz' - \bz)$, this condition is satisfied; thus, using $kt$ as the board dilation factor, we get a walk $k\bI := (k\bz=\bz^0,\bz^1,\ldots,\bz^l=k\bz'),$ from $k \bz$ to $k \bz'$ in $(kt\cBo)^q \cap \bbZ^{2q}$ consisting of integral multiples of moves $m_1$ and $m_2$.  
We normalize $k\bI$ to lie in $(t\cBo)^q \cap k\inv \bbZ^{2q}$ through division by $k$; this gives a walk 
$$
\bI = (\bz=k\inv\bz^0,k\inv\bz^1,k\inv\ldots,k\inv\bz^l=\bz')
\label{d:walk} 
$$ 
in $(t\cBo)^q \cap k\inv\bbZ^{2q}$ where each $k\inv\bz^i - k\inv\bz^{i-1} \in k\inv\bbZ^{2q}$ is a move of a single piece by an integral multiple of $k\inv m_1$ or $k\inv m_2$.

Define $\bI_0 := (0,k\inv\bz^1-\bz,k\inv\bz^2-\bz,\ldots,\bz'-\bz)$, the walk $\bI$ translated to the origin; thus we may write $\bI = \bz + \bI_0$.  

We squeeze $\bI$ into $\cR$ by shrinking and replicating it.  The line segment $[\bz,\bz']$ lies in $\cR \cap (t\cBo)^q$ by the convexity of regions and of the polytope $(t\cBo)^q$.  For some $\delta \in \bbR_{>0}$ the segment has a $\delta$-neighborhood $U$ contained in $\cR \cap (t\cBo)^q$.  By taking a sufficiently large divisor $\tau \in \bbZ_{>0}$ we can ensure that 
$\bz'' + \tau\inv\bI_0$ is contained in $U$ for every $\bz'' \in [\bz,\bz'-\tau\inv(\bz'-\bz)]$.  In particular, that is true for every $\bz'' = \bz + (j-1)\tau\inv (\bz'-\bz)$ such that $j \in \{1,\ldots,\tau\}$.  Consequently, the concatenated sequence 
$$
\bz + \tau\inv \bI_0, \ \bz + 1\tau\inv (\bz'-\bz) + \tau\inv \bI_0, \ \ldots, \ \bz + (\tau-1)\tau\inv (\bz'-\bz) + \tau\inv \bI_0
$$ 
is a walk from $\bz$ to $\bz'$ in $(t\cBo)^q \cap (k\tau)\inv\bbZ^{2q}$ by $(1/k\tau)$-fractions of legal moves, is contained in $\cR$, and is therefore a discrete isotopism from $\bz$ to $\bz'$.
\end{proof}

\begin{prop}\label{P:2piecetypes}
The number of nonattacking configuration types for $q=1$ is $1$.  For $q=2$ it is $|\M|$, the number of basic moves.
\end{prop}

\begin{proof}
Adding one piece at a time as described in the opening of this section, it is clear that two labelled pieces have $2|\M|$ configuration types and two unlabelled ones have $|\M|$.
\end{proof}

Thus, for any piece $\pP$, the formula for $u_\pP(q;n)$, upon setting $n=-1$, must yield $1$ when $q=1$ and $|\M|$ when $q=2$.  This provides a means for checking formulas.  Based on the (apparent) facts about queens and nightriders (see below), we propose:

\begin{conj}\label{Cj:3piecetypes}
The number of nonattacking configuration types of $3$ pieces depends only on $|\M|$.
\end{conj}

We think the number depends on the actual moves for large $q$, though we are not sure where that dependence begins; we suspect $q=4$ or $q=5$.

On the other hand, if there are hardly any basic moves, the number of configuration types is always independent of the actual moves.

\begin{thm}\label{T:2movetypes}
The number of unlabelled combinatorial types of nonattacking configuration of $q$ pieces is $1$ for a piece with only one move and $q!$ for any piece with exactly two moves.
\end{thm}

\begin{proof}
The case $|\M|=1$ is obvious.  Assume $\M=\{m_1,m_2\}$ and label the pieces.  
Take the basic moves as coordinate vectors in a new coordinate system, in which $m_1$ is the horizontal move to the right and $m_2$ is the vertical move up.

A configuration type is formalized by the list of lists just described.  
Since no two pieces are on the same horizontal line, the lists $L_{i1}$ are determined by the permutation of the piece labels in order of height.  Similarly, the lists $L_{i2}$ are determined by the order of piece labels from left to right.  There are $q!$ permutations of each kind, and the horizontal and vertical permutations are independent.  
Hence the number of labelled combinatorial types is $(q!)^2$.  Upon dividing by $q!$ we have $q!$ combinatorial types for unlabelled pieces.
\end{proof}

In particular, Theorem~\ref{T:2movetypes} applies to the rook $\pR$\label{d:RR} and the bishop $\pB$\label{d:BB} to give $q!$ unlabelled combinatorial configurations for $q$ pieces.  For the queen $\pQ$\label{d:QQ} and the nightrider $\pN$,\label{d:NN} we can apply Theorem~\ref{T:typenumber} to generate data.  For example, the methods from Part~II give the previously known formula for two queens,
$$
u_\pQ(2;n) = \frac{n^4}{2}-\frac{5 n^3}{3}+\frac{3 n^2}{2}-\frac{n}{3}.
$$
Subtracting the number of attacking pairs of squares in all knight-like diagonals from the total number of pairs gives the formula for two nightriders:
\begin{align*}%\label{E:}
u_\pN(2;n) 
&= \left\{\frac{n^4}{2}-\frac{5 n^3}{6}+\frac{3 n^2}{2}-\frac{11 n}{12}\right\} + (-1)^n \frac{n}{4}.
\end{align*}
For both of these equations, the number of combinatorial types of configuration of two (unlabelled) pieces is $u_\pQ(2;-1) = u_\pN(2;-1) = 4$, in accord with Proposition~\ref{P:2piecetypes}.

Using derived and known formulas for queens, we have the data in Table~\ref{Tb:queensCT}.
\begin{table}[h]
\begin{center}
\begin{tabular}{|r||c|}
\hline
\vstrut{12pt} Queens & Types \\
\hline\hline
\vstrut{12pt} $q =$ 1	& 1 	 \\
		\hline
\vstrut{12pt} 2		& 4 	 \\
		\hline
\vstrut{12pt} 3		& 36 	  \\
		\hline
\vstrut{12pt} 4		& 574*	\\
		\hline
\vstrut{12pt} 5		& 14206*  \\
		\hline
\vstrut{12pt} 6		& 501552$^\dagger$ \\
		\hline
\end{tabular}
\bigskip
\end{center}
\caption{The number of combinatorial configuration types of $q$ (unlabelled) queens in an $n \times n$ square board.
{\small 
\newline\hspace*{1em}* is a number deduced from a formula in \cite{ChMath}.  
\newline\hspace*{1em}${}^\dagger$ deduced from the formula of Karavaev (\cite{KaravaevWeb} and \cite[Sequence A176186]{OEIS}).
}}
\label{Tb:queensCT}
\end{table}

The number of combinatorial configuration types of nonattacking placements of $3$ nightriders in an $n \times n$ square board is $36$, based on \Kot's enormous formula for three nightriders (undoubtedly correct, though unproved) \cite{ChMath}.  The fact that this agrees with the number for three queens gave rise to Conjecture~\ref{Cj:3piecetypes}.  
Corollary~\Ctypestwothree\ on three partial queens supports the conjecture.

%=====
%%%%%%%%%%%%%%%%%%%%%%%%%%%%
\sectionpage
\section{Bounds on the period}\label{period}\

Theorem~\ref{T:formula} says nothing about the period of $u_\pP(q;n)$.  We want to bound the period by deriving the denominator from the plane geometry of $\cB$ (which gives the boundary inequalities) and from $\M$ (which gives the attack constraints).  

Let the boundary inequalities (with integral coefficients) of the polygon $\cB$ be $a_jx + b_jy \leq \beta_j$ for $1\leq j\leq \omega$.\label{d:bdy}  The pieces have coordinate (column) vectors $z_1, z_2, \ldots, z_q$, which must satisfy $a_jx_i+b_jy_i \leq \beta_j$ for all $1\leq i\leq q$ and $1\leq j \leq \omega$.  
Then the system $A\bz=\mathbf b$ in Equation~\eqref{E:grandmovesmatrix} contains all the equations that determine any one vertex of the inside-out polytope $(\cB^q,\cA_\pP)$.  
\begin{equation}
\label{E:grandmovesmatrix}
\begin{pmatrix}
M	& -M	& 0	& 0	&\cdots 	& 0	& 0	\\
M	& 0	& -M	& 0	&\cdots 	& 0	& 0	\\
\vdots&\vdots&\vdots&\vdots&\cdots	&\vdots&\vdots \\
M	& 0	& 0	& 0	&\cdots	& 0	& -M	\\
0	& M	& -M	& 0	&\cdots 	& 0	& 0	\\
0	& M	& 0	& -M	&\cdots 	& 0	& 0	\\
\vdots&\vdots&\vdots&\vdots&\cdots	&\vdots&\vdots \\
0	& M	& 0	& 0	&\cdots	& 0	& -M	\\
\vdots&\vdots&\vdots&\vdots&\cdots	&\vdots&\vdots \\
0	& 0	& 0	& 0	&\cdots	& M	& -M	\\ \hline
\vstrut{12pt}
B	& 0	& 0	& 0	&\cdots 	& 0	& 0	\\
0	& B	& 0	& 0	&\cdots 	& 0	& 0	\\
\vdots&\vdots&\vdots&\vdots&\cdots	&\vdots&\vdots \\
0	& 0	& 0	& 0	&\cdots	& 0	& B	\\
\end{pmatrix}\begin{pmatrix} z_1 \\ z_2 \\ \vdots \\ z_q \end{pmatrix}=
\begin{pmatrix}
0 \\ 0 \\ \vdots \\ 0 \\ 0 \\ 0 \\ \vdots \\ 0 \\ \vdots \\ 0 \\ \hline
\vstrut{12pt}
\bbeta \\ \bbeta \\ \vdots \\ \bbeta \\
\end{pmatrix},
\end{equation}
where $M$ and $B$ are the matrices
\[
M := \begin{pmatrix} m_1^\perp \\ m_2^\perp \\ \vdots \\ m_{|\M|}^\perp
\end{pmatrix}\qquad
B := \begin{pmatrix}
a_1 & b_1 \\ a_2 & b_2 \\ \vdots & \vdots \\ a_\omega & b_\omega \\
\end{pmatrix},
\]
containing the row vectors $m_r^\perp$, and where $\bbeta$ is the column vector of constant terms $\beta_1, \ldots, \beta_\omega$.  We define $A'$\label{d:A'} to be the top half of $A$.

A fundamental fact from linear algebra is the following lemma.

\begin{lem}\label{L:vertexgeo}
The configuration $\bz\in\bbR^{2q}$ is a vertex of the inside-out polytope $(\cP,\cA)$ if and only if it is in the closed polytope $\cP$ and there are $k$ attack equations and $2q-k$ boundary equations that uniquely determine $\bz$.
\end{lem}

A vertex corresponds to a set of violated boundary and attack constraints that determines uniquely (up to translation) a particular  placement of $q$ labelled pieces on the lattice points contained in some integer dilate $t\cB$ of the closed board. 

\subsection{Cramer's rule and rectangular boards}\label{cramer}\

Alternatively, we might investigate the system in Equation~\eqref{E:grandmovesmatrix} directly as a matrix.  It follows by Cramer's Rule and Lemma~\ref{L:vertexgeo} that every denominator of an inside-out vertex divides a $2q \times 2q$ subdeterminant of $A$.  
The period of the counting quasipolynomial of $(\cB^q,\cA_\pP)$ is a divisor of the least common multiple of all $2q \times 2q$ subdeterminants of $A$.  This quantity is not so easy to determine, but for rectangular boards there is a way to estimate it, if there are not too many moves.

For a rational rectangular board with sides on the axes, say for instance $\cB = [0,a] \times [0,b]$, the dilation $(n+1) \cBo$ contains integral points in an $(na-1)\times(nb-1)$ rectangle.  (This is not precisely what one wants of a rectangular board; the proportions should remain fixed under dilation.  However, it is what our method handles.)  
The augmented matrix is 
$$
\begin{pmatrix} B & \beta \end{pmatrix} = \begin{pmatrix} -1 & 0 & 0 \\ 0 & -1 & 0 \\ 1 & 0 & a \\ 0 & 1 & b \\ \end{pmatrix}.
$$
Rearranging the bottom half of $A$ in \eqref{E:grandmovesmatrix}, it takes the form 
$\begin{pmatrix} -I_{2q} \\ I_{2q} \end{pmatrix}.$\label{d:I}   
Consequently, for a rectangular board the values of the $2q \times 2q$ subdeterminants of $A$ are the values of the subdeterminants of any order of $A'$, the top half of $A$.  Thus, the period of the counting quasipolynomial divides $\lcmd(A')$, the least common multiple of all subdeterminants of $A'$.  
The value of $\lcmd(A')$ is the only general theoretical bound we know for the period without finding the denominator itself. 

In general $\lcmd(A')$ is difficult to compute.  Two of us studied it and found a computable formula that applies as long as the moves matrix $M$ has up to two rows \cite{DKP}. Observe that $A'$ is the Kronecker product $\Eta_q\transpose \otimes M$ where $\Eta_q\transpose$\label{d:Eta} ($\Eta$ is `Eta') is the matrix consisting of one row for each pair of different pieces, say $i$ and $j$, in which all columns are zero except for a $1$ in the column of piece $i$ and a $-1$ in that of piece $j$.  Thus, $\Eta_q$ is the oriented incidence matrix of the complete graph $K_q$,\label{d:Kq} which is well known to be totally unimodular with rank $q-1$.  In this situation we can calculate $\lcmd(A')$ when $M$ has two rows by means of \cite[Corollary 2]{DKP}, which in terms of $M = \begin{pmatrix} d_1 & -c_1 \\ d_2 & -c_2 \end{pmatrix}$ and $\Eta_q\transpose$ states that, for a piece with two moves, 
\begin{equation}\label{E:lcmd2}
\lcmd(\Eta_q\transpose \otimes M) = \lcm\bigg((\lcmd M)^{q-1}, \LCM_{p=1}^{\lfloor q/2 \rfloor} \Big| \det\begin{pmatrix} d_1^p&c_1^p \\ d_2^p&c_2^p \end{pmatrix}\Big|^{\lfloor q/2p \rfloor}\bigg),
\end{equation}
where $\LCM_p$ denotes the least common multiple of all the terms for $p$ in the indicated range.  We conclude that, for a piece with one or two moves on a rectangular board with sides parallel to the axes, the period of $o_\pP(q;n)$ is a divisor of the right-hand side of Equation~\eqref{E:lcmd2}.  This applies, for example, to the bishop, where the right-hand side equals $2^q$ (see \cite{DKP}).

Unfortunately, this bound on the period is far from sharp (see Section~\ref{lcmdex}) and, worse, the theory of \cite{DKP} does not apply to a matrix $M$ with more than two rows, which means a piece with more than two move directions.  For such matrices, e.g., for the queen and nightrider, we have to calculate the determinantal upper bound $\lcmd(A')$ on the period separately for each value of $q$.  One hopes that Equation~\eqref{E:lcmd2} can be generalized to $m \times 2$ matrices, though probably it is excessively complex when there are more than two moves.

In summary, the disadvantage of the $\lcmd$ bound is that it is weak; the advantage is that it is explicit if $\pP$ has one or two moves.

\subsection{True periods and theoretical bounds}\label{lcmdex}\

For a rider with the single move $(c,d)$, the bound analogous to \eqref{E:lcmd2} is $(\lcmd M)^{q-1} = |\lcm(c,d)|^{q-1}$, which is $1$ if $cd=0$ and otherwise $|cd|^{q-1}$.  The true periods for $q=1,2$ are $1$ and $\max(|c|,|d|)$, respectively, as shown in Part~II. 
For move $(c,d)=(1,2)$, the true periods for $q\leq3$ are respectively 1, 2, 2 while the bounds are 1, 2, 4.

For the bishop, the move matrix has $\lcmd(M_\pB) = 2$. It follows that the period of $u_\pB(q;n)$ divides $2^{q-1}$.  
In Part~V we prove that an upper bound on the period of $u_\pB(q;n)$ is 2, which rigorously establishes the period and consequently the correctness of \Kot's quasipolynomial formulas.  The proof relies on signed graph theory applied to the bishops hyperplane arrangement $\cA_\pB$. 
Table~\ref{Tb:BQN} shows descriptive data for nonattacking placements of few bishops.
\begin{table}[htb]
\begin{center}
\begin{tabular}{|r||c|c|c|}
\hline
\vstrut{12pt} Bishop	 & Period & Denom	& lcmd \\
\hline\hline
\vstrut{12pt} $q =$ 1		& 1 	& 1 	& 1 \\
		\hline
\vstrut{12pt} 2	 	& 1	& 1 & 2 \\
		\hline
\vstrut{12pt} 3		& 2 & 2 & 4 \\
		\hline
\vstrut{12pt} 4		& 2	& 2 & 8 \\
		\hline
\vstrut{12pt} 5		& 2	& 2 & 16 \\
		\hline
\vstrut{12pt} 6		& 2	& 2 & 32 \\
		\hline
\end{tabular}
\qquad
\begin{tabular}{|r||c|c|c|}
\hline
\vstrut{12pt}  Queen &  Period & Denom.\ & lcmd \\
\hline\hline
\vstrut{12pt} $q =$ 1	 & 1 	& 1	& 1 \\
		\hline
\vstrut{12pt} 2		& 1 	& 1	& 2 \\
		\hline
\vstrut{12pt} 3		& 2 	& 2	& 4  \\
		\hline
\vstrut{12pt} 4		& 6*	&6	& 24 \\
		\hline
\vstrut{12pt} 5		& 60*	& ---	& ---  \\
		\hline
\vstrut{12pt} 6		& 840$^\dagger$ & --- & --- \\
		\hline
\vstrut{12pt} 7		& 360360$^\ddagger$ & ---	& --- \\
		\hline
\end{tabular}
\bigskip

\begin{tabular}{|c||c|c|c|}
\hline
\vstrut{12pt}  Nightrider  & Period & Denom & lcmd \\
\hline\hline
\vstrut{12pt} $q =$ 1		& 1 	& 1	& 1 \\
		\hline
\vstrut{12pt} 2			& 2	& 2	& 60 \\
		\hline
\vstrut{12pt} 3			& 60*	& 60	& 3600 \\
		\hline
\vstrut{12pt} 4		& ---	&14559745200	& 14290972303608000 \\
		\hline
\end{tabular}
\bigskip
\end{center}
\caption{The period of the counting quasipolynomial for $q$ bishops, queens, or nightriders in an $n \times n$ square board, the denominator of the inside-out polytope, and (``lcmd'') the determinantal upper bound on the period.  Periods without denominators are unproved.
{\small 
\newline\hspace*{1em}* is a number deduced from a formula in \cite{ChMath}.  
\newline\hspace*{1em}${}^\dagger$ is deduced from the formula of Karavaev (\cite{KaravaevWeb} and \cite[Sequence A176186]{OEIS}).
\newline\hspace*{1em}${}^\ddagger$ is deduced from the generating function in \cite[Sequence A178721]{OEIS}.}  
}
\label{Tb:BQN}
\end{table}

Unlike in the case of bishops, the period of the counting quasipolynomial $u_\pQ(q;n)$ for $q$ queens is not simple and we have no general formula.  The denominator of the inside-out polytope and the value of $\lcmd(M)$ can only be computed for very small values of $q$.  Again we see that $\lcmd(M)$ is a weak bound.  Table~\ref{Tb:BQN} collects the known and conjectured periods for queens.

For the nightrider, the move matrix has $\lcmd(M_\pN) = 60$ \cite[Example 3]{DKP}, thus giving the lcmd bound in the table.   We calculated the denominator directly using Mathematica.  The difference is substantial.

%=====
%%%%%%%%%%%%%%%%%%%%%%%%%%%%
\sectionpage\section{Questions, Extensions}\label{last}

Work on nonattacking chess placements raises many questions, several of which have general interest.  
We propose the following questions and directions, with others to come in subsequent parts of this series.

%=====
\subsection{Combinatorial configuration types}\label{combconfigtypes}\

We noticed that three queens and three nightriders have the same number of combinatorial types of nonattacking configuration.  The queen and nightrider also have the same number of moves.

\begin{quest}\label{Q:combconfigtypes}
Does the number of combinatorial types of nonattacking configuration depend only on the number of moves?  If so, what is the formula?
\end{quest}

%=====
\subsection{The slope matroid}\label{slope}\

The \emph{slope matroid} $\SM_q(\pP)$\label{d:SM} of order $q$ of a piece $\pP$ with basic move set (or for this purpose, slope set) $\M$ is the matroid of the move arrangement $\cA_\pP^q$ that consists of all move hyperplanes in $\bbR^{2q}$.  It should be viewed as a matroid on the edge set of the slope graph $\S_q(\pP)$.  The problem is to describe the rank function and closed sets of the slope matroid.  

We propose that the rules for closed sets are the same as the geometrical incidence theorems about rational points and slopes.  We cannot say exactly what that means, but here is an example.  Suppose we have slopes $1/0, 0/1, 1/1, -1/1$.  Consider the hyperplanes $\cH^{0/1}_{12}, \cH^{0/1}_{34}, \cH^{1/0}_{13}, \cH^{1/0}_{24}, \cH^{1/1}_{14}$.  They force $z_1, z_2, z_4, z_3$ to be the corners of a square and consequently we get $\cH^{-1/1}_{23}$ in their closure due to the necessary incidences of two pairs of parallel lines and their 45${}^\circ$ diagonals.

The ultimate goal is to automate the listing of closed subgraphs of $\SM_q(\M)$.  Since the automorphism groups and M\"obius function can be computed automatically without too much difficulty, that would enable automatic generation of formulas for $u_\pP(q;n)$ for arbitrary sets of moves and large values of $q$.  Since that goal requires knowing all rational incidence theorems it is unlikely to be attainable except for relatively small $q$ and $\M$, but $\M$ indeed is small for real pieces, and any understanding of small incidences would enlarge the range of accessible values of $q$.

%=====
\subsection{Riders versus non-riders }\

\Kot's many formulas are quasipolynomials only for riders.  For all others he gets an eventual polynomial, as in our analysis of pieces on a $k\times n$ board where $k$ is fixed \cite{QRS}.  It seems clear that the reason he does not get a quasipolynomial is that, with nonriders, not all moves have unbounded distance, so Ehrhart theory does not apply.  The reason he gets an eventual polynomial is less apparent.  We believe it is, in essence, that the count is the number of ways to place a finite number of ``tight'' nonattacking configurations involving a total of $q$ pieces so that no two tight configurations overlap, each tight configuration that can fit on the board contributes a polynomial to the total count, and for large $n$ the board is big enough that every possible tight configuration can fit.  How to make this intuitive statement precise is not precisely clear.

%=====
\subsection{Varied moves }\label{varied}\

Our counting method extends to a much more general situation.  For convenience we assume distinguishable pieces, $\pP_1, \ldots, \pP_q$.  Think of the moves as attacks, and suppose the basic attacks $m_{ij,r}$ may depend on both the attacking piece $\pP_i$ and the attacked piece $\pP_j$.  This may seem unrealistically general but it permits us to combine more than one interesting type of situation.  We form a move matrix $M_{ij}$ from the basic attacks of $\pP_i$ on $\pP_j$.  Theorem~\ref{T:formula} and the ensuing discussion of the period remains valid if we take $A'$ (the upper half of the system in Equation~\eqref{E:grandmovesmatrix}) to be the matrix in Equation~\eqref{E:grandattacksmatrix}.
\begin{figure}[hbt]
\begin{equation}
\label{E:grandattacksmatrix}
A' = \begin{pmatrix}
M_{12}	& -M_{12}	& 0	& 0	&\cdots 	& 0	& 0	\\
M_{21}	& -M_{21}	& 0	& 0	&\cdots 	& 0	& 0	\\
M_{13}	& 0	& -M_{13}	& 0	&\cdots 	& 0	& 0	\\
M_{31}	& 0	& -M_{31}	& 0	&\cdots 	& 0	& 0	\\
\vdots&\vdots&\vdots&\vdots&\cdots	&\vdots&\vdots \\
M_{1q}	& 0	& 0	& 0	&\cdots	& 0	& -M_{1q}	\\
M_{q1}	& 0	& 0	& 0	&\cdots	& 0	& -M_{q1}	\\
0	& M_{23}	& -M_{23}	& 0	&\cdots 	& 0	& 0	\\
0	& M_{32}	& -M_{32}	& 0	&\cdots 	& 0	& 0	\\
0	& M_{24}	& 0	& -M_{24}	&\cdots 	& 0	& 0	\\
0	& M_{42}	& 0	& -M_{42}	&\cdots 	& 0	& 0	\\
\vdots&\vdots&\vdots&\vdots&\cdots	&\vdots&\vdots \\
0	& M_{2q}	& 0	& 0	&\cdots	& 0	& -M_{2q}	\\
0	& M_{q2}	& 0	& 0	&\cdots	& 0	& -M_{q2}	\\
\vdots&\vdots&\vdots&\vdots&\cdots	&\vdots&\vdots \\
0	& 0	& 0	& 0	&\cdots	& M_{q-1,q}	& -M_{q-1,q}	\\
0	& 0	& 0	& 0	&\cdots	& M_{q,q-1}	& -M_{q,q-1}	\\
\end{pmatrix} .
\end{equation}
\end{figure}

The most realistic case is that where, as in chess, the moves (or attacks) do not depend on the piece being attacked.  In that case, $M_{ij} = M_i$, independent of $j$, and the matrix $A'$ becomes more similar to that of Equation~\eqref{E:grandmovesmatrix}.

%=====
\subsection{Higher dimensions }\

It is tempting to apply the inside-out polytope method to boards of higher dimension such as hypercubical boards $\cB = (0,1)^d$.   However, pieces with multidimensional moves would surely be much more difficult to treat.  For two-dimensional moves $m_r$, the orthogonal vector $m_r^\perp$ defines the move line so the attacking configurations in $\bbR^{dq}$ are determined by a hyperplane; but when $d>2$ a move line requires more than one equation to define it, so the attacking configurations are determined by a subspace of codimension $d-1$.

%=====
\subsection{A generalization of total dual integrality?}\

The least common multiple of subdeterminants of the coefficient matrix of the attack hyperplanes (that is, $\lcmd$) turned out to be a very inefficient bound on the period, because it is much larger than the least common denominator of all vertices.  This reminds us of the fact that there are totally dual integral matrices which are not totally unimodular; indeed the analogy is close, since total unimodularity means the $\lcmd = 1$.  We suggest that a worthy general question about an integral $r \times s$ matrix $M$ is the relationship between $\lcmd M$ and the least common denominator $D$ of all lattice vertices of $M$, defined as points $\bz \in \bbR^s$ determined by restrictions $A\bz \in \bbZ^s$ where $A$ is any nonsingular matrix consisting of $s$ rows of $M$.  Though $D$ may usually be much less than $\lcmd M$, the cases of equality, being analogs of totally unimodular matrices, might be quite interesting.

%%%%%%%%%%%%%%%%%%%%%%%%%%%%%%%%%%%%%%%%%%%%%%%%%%%%%%%%
\bigskip\bigskip

\sectionpage\section*{Dictionary of Notation} 
\medskip

\hspace{-.25in}
\begin{minipage}[t]{3.35in}\vspace{0in}
%Latin letters
\begin{enumerate}[]
\item $a,b$ -- dimensions of rectangle (p.\ \pageref{d:a}) 
\item $a_j, b_j$ -- coeffs of $\cB$ boundary ineq (p.\ \pageref{d:bdy})
\item $(c,d),(c_r,d_r)$ -- coords of basic move (p.\ \pageref{d:cd1})
\item $d/c$ -- slope of a line (p.\ \pageref{d:cd2})
\item $d$ -- degree of quasipolynomial (p.\ \pageref{d:f})
\item $d$ -- dimension of polytope (p.\ \pageref{d:f})
\item $e_j$ -- coefficient of quasipolynomial (p.\ \pageref{L:iop})
\item $f(t)$ -- quasipolynomial function (p.\ \pageref{d:f})
\item $f_k(t)$ -- constituent of quasipolynomial (p.\ \pageref{d:f})
\item $m_r = (c_r,d_r), m=(c,d)$ -- basic move (p.\ \pageref{d:mr})
\item $m_r^\perp$ = $(d_r,-c_r)$ -- orthogonal to $m_r$ (p.\ \pageref{d:mrperp})
\item $n+1$ -- dilation factor for board (p.\ \pageref{d:n}) 
\item $o_\pP(q;n)$ -- \# nonattacking lab configs (p.\ \pageref{d:nonattacking})
\item $p$ -- period of quasipolynomial (p.\ \pageref{d:f})
\item $q$ -- \# pieces on a board (p.\ \pageref{d:q})
\item $r$ -- move index (p.\ \pageref{d:r})
\item $t$ -- dilation (inflation) variable (p.\ \pageref{d:t})
\item $u_\pP(q;n)$ -- \# nonattack unlab configs (p.\ \pageref{d:nonattacking})
\item $z=(x,y)$, $z_i=(x_i,y_i)$ -- piece position (p.\ \pageref{d:zi})
\end{enumerate}
\medskip
\begin{enumerate}[]
\item $\bz = (z_1,\ldots,z_q)$ -- configuration (p.\ \pageref{d:bfz})
\end{enumerate}
\medskip
%l.c. Greek
\begin{enumerate}[]
\item $\alpha(\cU;n)$ -- attacking config count (p.\ \pageref{d:alphaU})
\item $\beta_j$ -- constant in $\cB$ boundary ineq  (p.\ \pageref{E:grandmovesmatrix})
\item $\gamma_i$ -- coefficient in $u_\pP$ (p.\ \pageref{d:gamma})
\item $\kappa$ -- \# pieces involved in subspace (p.\ \pageref{d:kappa})
\item $\mu$ -- M\"obius function of $\cL(\cP^\circ,\cA)$ (p.\ \pageref{d:mu})
\item $\nu$ -- codim of subspace (p.\ \pageref{d:codim})
\item $\omega$ -- \# boundary lines of $\cB$ (p.\ \pageref{E:grandmovesmatrix})
\end{enumerate}
\medskip
%Capital Greek
\begin{enumerate}[]
\item $\bar\Gamma_j$ -- linear quasipolynomial (p.\ \pageref{d:bG})
\item $\Eta_q$ -- incidence matrix of $K_q$ (p.\ \pageref{d:Eta})
\item $\S, \S_q$ -- slope graph (p.\ \pageref{d:S})
\end{enumerate}
\bigskip\bigskip
%Abbreviations
\begin{enumerate}[]
\item $\Aut(\cU)$ -- subspace automorphisms (p.\ \pageref{d:Aut})
\item $\codim(\cU)$ -- subspace codimension 
\item $\dim(\cU)$ -- subspace dimension
\item $\SM$ -- slope matroid (p.\ \pageref{d:SM})
\item $\vol(\cU\cap\cP)$ -- polytope volume (p.\ \pageref{d:vol})
\end{enumerate}

\end{minipage}
\begin{minipage}[t]{3.5in}\vspace{0in}
%Uppercase Roman
\begin{enumerate}[]
\item $A$ -- grand matrix in Eq.~\eqref{E:grandmovesmatrix} (p.\ \pageref{E:grandmovesmatrix})
\item $A'$ -- matrix of eqns of move hyps (p.\ \pageref{d:A'})
\item $B$ -- matrix of coeffs of $\cB$ bdry lines (p.\ \pageref{E:grandmovesmatrix})
\item $C$ -- matrix of two moves (p.\ \pageref{d:C})
\item $D$ -- denom of (inside-out) polytope (p.\ \pageref{d:D})
\item $E_{\cP}$ -- Ehrhart quasipoly (p.\ \pageref{d:Ehr})
\item $E_{\cP}^\circ$ -- open Ehrhart quasipoly (p.\ \pageref{d:E})
\item $E_{\cP,\cA}^\circ$ -- open Ehrhart of inside-out (p.\ \pageref{d:Eiop})
\item $I$ -- identity matrix
\item $K_q$ -- complete graph (p.\ \pageref{d:Kq})
\item $L_{ir}$ -- list in configuration (p.\ \pageref{d:Lir})
\item $M$ -- matrix of moves (p.\ \pageref{E:grandmovesmatrix})
\item $N$ -- variable counting lattice points (p.\ \pageref{d:N})
\end{enumerate}
\medskip
%Other
\begin{enumerate}[]
\item $\bI, \bI_0$ -- walk in configuration space (p.\ \pageref{d:walk})
\item $\M$ -- set of basic moves (p.\ \pageref{d:moveset})
\end{enumerate}
\medskip
%Geometry
\begin{enumerate}[]
\item $\cA$ -- arrangement of hyperplanes (p.\ \pageref{d:cA})
\item $\cA_{\pP}$ -- move arrangement of piece $\pP$ (p.\ \pageref{d:AP})
\item $\cB, \cBo$ -- closed, open board polygon (p.\ \pageref{d:B})
\item $\Coeff_k(f(n))$ -- coefficient of $n^k$ in $f(n)$ (p.\ \pageref{d:cC})
\item $\cH_{ij}^{d/c}$ -- hyperplane for move $(c,d)$ (p.\ \pageref{slope-hyp})
\item $\cH_{ij}^m$ -- hyperplane for move $m$ (p.\ \pageref{slope-hyp})
\item $\cL$ -- intersection semilattice  (p.\ \pageref{d:cL})
\item $\cP$, $\cP^\circ$ -- polytope, open polytope (p.\ \pageref{d:cP})
\item $(\cP,\cA)$ -- inside-out polytope (p.\ \pageref{d:cP})
\item $\cR$ -- region of arr or inside-out poly (p.\ \pageref{d:cR})
\item $\cU$ -- subspace in intersection semilatt (p.\ \pageref{d:U})
\item $\tcU$ -- essential part of $\cU$ (p.\ \pageref{d:tcU})
\item $\cW_{ij\ldots}^{\,d/c}$ -- subspace of slope relation (p.\ \pageref{d:Wdc})
\item $\cW_{ij\ldots}^{\,=}$ -- subspace of equal position (p.\ \pageref{d:W=})
\end{enumerate}
\medskip
\begin{enumerate}[]
\item $\bbQ$ -- rational numbers
\item $\bbR$ -- real numbers
\item $\bbR^{2q}$ -- configuration space (p.\ \pageref{d:configsp})
\item $\bbZ$ -- integers
\end{enumerate}
\medskip
	%Pieces
\begin{enumerate}[]
\item $\pB$ -- bishop (p.\ \pageref{d:BB})
\item $\pN$ -- nightrider (p.\ \pageref{d:NN})
\item $\pP$ -- piece (p.\ \pageref{d:P})
\item $\pQ$ -- queen (p.\ \pageref{d:QQ})
\item $\pR$ -- rook (p.\ \pageref{d:RR})
\end{enumerate}
\end{minipage}

\newpage

%%%%%%%%%%%%%%%%%%%%%%%%%%%%%%%%%%%%%%%%%%%%%%%%%%%%%%%
\newcommand\otopu{$\overset{\circ}{\textrm u}$}

\end{document}